\documentclass[11pt, reqno]{amsart}

\usepackage{amsmath,amssymb,amsfonts,euscript,enumerate,amsthm}
\usepackage{color}

\usepackage{a4wide}
\usepackage{graphicx}
\usepackage{epstopdf}
\usepackage[colorlinks=true,citecolor=red,linkcolor=blue]{hyperref}
\usepackage[]{cleveref}
\usepackage{mathtools}
\usepackage{esint}
\usepackage{mathrsfs}
\usepackage[notref, notcite, final]{showkeys}

\theoremstyle{plain}
\newtheorem{theorem}{Theorem}[section]
\newtheorem{lemma}[theorem]{Lemma}

\newtheorem{proposition}[theorem]{Proposition}

\theoremstyle{definition}
\newtheorem{definition}[theorem]{Definition}

\theoremstyle{remark}
\newtheorem{remark}[theorem]{Remark}

\numberwithin{equation}{section}

\newcommand{\R}{{\mathbb R}}
\newcommand{\bS}{\mathbb{S}}

\newcommand{\al}{\alpha}

\newcommand{\de}{\delta}
\newcommand{\e}{\varepsilon}

\newcommand{\la}{\lambda}
\newcommand{\si}{\sigma}

\newcommand{\Si}{\Sigma}
\newcommand{\Om}{\Omega}

\newcommand{\Ga}{\Gamma}

\newcommand{\pa}{\partial}

\newcommand{\qu}{\quad}

\newcommand{\ra}{\rightarrow}

\newcommand{\D}{\nabla}
\newcommand{\De}{\Delta}

\newcommand{\fr}{\frac}

\newcommand{\inn}[2]{\left\langle {#1},{#2} \right\rangle}
\newcommand{\norm}[1]{\left\lVert#1\right\rVert}

\makeatletter
\@namedef{subjclassname@2020}{\textup{2020} Mathematics Subject Classification}
\makeatother

\title[An eigenvalue problem for prescribed curvature equations]{An eigenvalue problem for prescribed curvature equations}

\author{Taehun Lee}
\address{School of Mathematics, Korea Institute for Advanced Study, Seoul 02455, Korea}
\email{taehun@kias.re.kr}

\subjclass[2020]{53C42 (Primary) 35J60 35P30 (Secondary)}
\keywords{curvature equation, eigenvalue problem, uniqueness}

\begin{document}


\begin{abstract}	
We study an eigenvalue problem for prescribed $\sigma_k$-curvature equations of star-shaped, $k$-convex, closed hypersurfaces. We establish the existence of a unique eigenvalue and its associated hypersurface, which is also unique, provided that the given data is even. Moreover, we show that the hypersurface must be strictly convex. A crucial aspect of our proof involves deriving uniform estimates in $p$ for $L_p$-type prescribed curvature equations.
\end{abstract}

\maketitle

%
%
\section{Introduction}
Let $\Si$ be a smooth, closed hypersurface in $\R^{n+1}$. We use $\nu(X)$ and $\kappa(X)$ to denote the outward unit normal and principal curvatures of the hypersurface $\Si$ at point $X$, respectively. 
This paper concerns the following eigenvalue problem for prescribed curvature equations:
\begin{align}\label{eq:main-X}
\inn{X}{\nu(X)}^{k}\si_k(\kappa(X))=\la\psi(\nu(X)) \qu \text{for any }X\in \Si,
\end{align}
where $\si_k:\R^n \ra \R$ represents the elementary symmetric polynomial of degree $k$ and $\psi$ is a smooth, positive function on $\bS^n$. The curvature $\si_k(\kappa)$ includes the mean curvature, scalar curvature, and Gauss curvature for cases when $k=1,2$, and $n$, respectively. 

Equation \eqref{eq:main-X} is connected to singularity models for an anisotropic flow governed by powers of curvature $\si_k$, see \cite{BCD17_Acta,DH22_CVPDE,GLM18_NoDEA,Huisken90_JDG,McCoy11_ASNSPCS}. Indeed, for $p\not=1$, if $X$ denotes a parametrization of the evolving hypersurfaces $\Si_t$ satisfying
\begin{align*}
\fr{\pa }{\pa t}X= -\operatorname{sign}(1-p)(f(\nu)\si_k(\kappa))^{\fr{1}{1-p}}\nu,
\end{align*}
then its self-similar solutions fulfill
\begin{align}\label{eq:Lp}
\inn{X}{\nu(X)}^{p-1} \si_k(\kappa(X))  =   \la /f(\nu(X))
\end{align}
for some constant $\lambda>0$. 
Therefore, \eqref{eq:main-X} corresponds to \eqref{eq:Lp} with $p=k+1$ and the anisotropic function $f=1/\psi$. In particular, for the case of $k=1$ and $p=2$, equation \eqref{eq:Lp} represents a self-similar solution of the (anisotropic) inverse mean curvature flow \cite{CD21_Duke,DH22_CVPDE}. Note that \eqref{eq:Lp} remains invariant under dilation of the hypersurface if and only if $p=k+1$, and the constant $\la$ in \eqref{eq:Lp} can be absorbed by scaling when $p\not=k+1$.

In the case of $k=n$, equation \eqref{eq:Lp} becomes the celebrated $L_p$ Minkowski problem initiated by Lutwak \cite{Lutwak93_JDG}, where the admissible solutions are convex. This is a central problem in convex geometry, and there is a vast literature on the $L_p$ Minkowski problem; see for instance \cite{BLYZ13_JAMS,CW06_AM,  GLW22_arxiv, LO95_JDG, Schneider14_book} and the references therein. In terms of the support function $u:\bS^n\ra \R$ defined by $u(x)=\inn{x}{\nu^{-1}(x)}$ for a strictly convex hypersurface $\Si$, equation \eqref{eq:Lp} can be transformed into 
\begin{align}\label{eq:Lp-u}
\det(\D^2 u + u I) = \la^{-1}u^{p-1}f \qu \text{on } \bS^n,
\end{align}
where $\D$ denotes the covariant derivative with respect to an orthonormal frame on $\bS^n$ and $I$ is the identity matrix. The existence results for \eqref{eq:Lp-u} can be found in \cite{CW06_AM} for $p>-n-1$ and \cite{GLW22_arxiv} for $p<-n-1$. In particular, Chou and Wang \cite{CW06_AM} completely resolved the eigenvalue problem in this case ($p=n+1$ in \eqref{eq:Lp-u} or $k=n$ in \eqref{eq:main-X}) by obtaining a unique pair $(u,\la)$, where $u$ is the support function of a strictly convex, closed hypersurface that is positive on $\bS^n$, and $\la>0$ is the eigenvalue.

On the other hand, when $p=1$, equation \eqref{eq:Lp} reduces to the well-known prescribed curvature equation. In \cite{GG02_Annals}, Guan and Guan demonstrated the existence of a strictly convex, closed hypersurface satisfying \eqref{eq:Lp}, under the condition that $f$ is even, that is, $f(x)=f(-x)$ for all $x\in \bS^n$.

We note that an equation analogous to \eqref{eq:Lp} arises in the context of curvature measures, which can be expressed as follows:
\begin{align}\label{eq:curv-measure}
\si_k(\kappa(X))=\fr{\inn{X}{\nu(X)}}{|X|^{n+1}}f\left(\fr{X}{|X|}\right).
\end{align}
Given certain conditions on $f$, Guan, Lin, and Ma \cite{GLM09_IMRN} demonstrated the existence of convex solutions to this equation. Moreover, Guan, Li, and Li \cite{GLL12_Duke} proved the existence of a unique admissible solution (which is $k$-convex and star-shaped, see \Cref{def:admissible} below) for a general $f$. The $L^p$ dual generalization of this problem, in the measure form of \eqref{eq:curv-measure}, has been further investigated in \cite{BHP18_JDG,LYZ18_AM}.\\

In this paper we are interested in the eigenvalue problem \eqref{eq:main-X} for all $1\le k< n$, assuming that $\psi$ is an even function, that is, $\psi(x)=\psi(-x)$ for all $x\in \bS^n$.
To describe an admissible hypersurface $\Si$ of the eigenvalue problem \eqref{eq:main-X}, it is necessary to define the notions of a $k$-convex and a star-shaped hypersurface.

\begin{definition}\label{def:admissible}
Let $\Si$ be a $C^2$ regular hypersurface in $\mathbb{R}^{n+1}$.
\begin{enumerate}[(i)]
\item A hypersurface $\Si$ is said to be \textit{$k$-convex} if its principal curvatures $\kappa(X)=(\kappa_1(X),\ldots,\kappa_n(X))$ lie in $\Gamma_k$ for all $X \in \Si$. Here, $\Ga_k$ represents the G$\mathring{a}$rding cone, defined as
\begin{align*}
\Gamma_k=\left\{\lambda \in \mathbb{R}^n \mid \sigma_i(\lambda)>0, i=1, \ldots, k\right\}.
\end{align*}
\item A hypersurface $\Si$ is called \textit{star-shaped} if $\inn{X}{\nu(X)}>0$ for all $X\in \Si$.
\end{enumerate}
\end{definition}

We say that a hypersurface $\Si$ is \textit{admissible} if $\Si$ is a $k$-convex, star-shaped hypersurface. The number $\lambda$ is referred to as the eigenvalue of \eqref{eq:main-X} if there exists an admissible hypersurface $\Si$ that satisfies \eqref{eq:main-X} for the given $\lambda$. We note that \eqref{eq:main-X} is ensured to be elliptic at all admissible hypersurfaces. 

We now state the main result of this paper. 

\begin{theorem}\label{thm:main}
Let $1\le k< n$, and let $\psi$ be a smooth, positive, even function on $\bS^n$. Then, there exists a unique eigenvalue $\la>0$ for \eqref{eq:main-X}. Moreover, the corresponding admissible hypersurface $\Si$ is strictly convex and unique up to dilation.
\end{theorem}

\begin{remark}
\begin{enumerate}[(i)]
\item From the uniqueness, it can be deduced that the hypersurface $\Si$ is origin-symmetric. 
\item The evenness of $\psi$ will be used in establishing the existence of an eigenvalue and hypersurface.
The uniqueness of the eigenvalue and hypersurface does not require the evenness of $\psi$. See the proof of \Cref{thm:main} in \Cref{sec:pf}.
\item As previously noted, the case when $k=n$ was resolved for general $f$ in \cite{CW06_AM}.
\end{enumerate}
\end{remark}

We emphasize that, unlike the Laplace operator which possesses an infinite number of eigenvalues, our problem \eqref{eq:main-X} has only one eigenvalue. Furthermore, any origin-symmetric admissible solution to \eqref{eq:main-X} must be strictly convex. 

\bigskip

The proof of the main theorem relies on a compactness argument by deriving uniform estimates in $p$ for \eqref{eq:Lp} with $k+1<p<k+1+\e$. Additionally, in the same range of $p$, the existence and uniqueness of \eqref{eq:Lp} are required, where the existence follows from the work of Guan, Ren, and Wang \cite{GRW15_CPAM} (see \Cref{prop:exist} below). We will prove the uniqueness in \Cref{prop:unique}. Combining existence and uniqueness, we summarize the result as follows.
Here, $f$ is not necessarily even. An admissible solution to \eqref{eq:Lp} is defined in the same manner, imposing conditions that the hypersurface is $k$-convex and star-shaped. 

\begin{theorem}\label{thm:2}
Suppose $p>k+1$ with $1\le k\le n$, and let $f$ be a smooth, positive function on $\bS^n$. Then, there exists a unique admissible solution to \eqref{eq:Lp} for any $\la>0$. Moreover, the solution must be strictly convex.
\end{theorem}

We note that the quantity $\si_k(\kappa)$ corresponds to an operator in Hessian quotient equation when expressed as a support function. Indeed, if $\Si$ is strictly convex, then in terms of the support function $u$ of $\Si$, \eqref{eq:Lp} becomes
\begin{align}\label{eq:Lp-u-intro}
\fr{\det(\D^2 u+uI)}{\si_{n-k}(\D^2u+uI)} = \fr{fu^{p-1}}{\lambda} \qu \text{on }\bS^n,
\end{align}
as described in \eqref{eq:main-u} below. When $k=n$, we can see that this equation reduces to the $L_p$ Minkowski problem \eqref{eq:Lp-u}, as mentioned earlier. 

Another important equation that includes \eqref{eq:Lp-u} is the equation for the $L_p$ Christoffel--Minkowski problem, which is formulated as
\begin{align}\label{eq:CM}
\si_k(\D^2 u + uI) = \fr{fu^{p-1}}{\la}\qu \text{on } \bS^n.
\end{align}
See Bryan--Ivaki--Scheuer \cite{BIS23_TAMS}, Guan--Lin--Ma \cite{GLM06_CAMSB}, Guan--Ma \cite{GM03_Invent}, and Guan--Ma--Zhou \cite{GMZ06_CPAM} for the case $p=1$; Guan--Xia \cite{GX18_CVPDE} for $1<p<k+1$; Hu--Ma--Shen \cite{HMS04_CVPDE} for $p\ge k+1$.

When $p=k+1$, a key distinction between equations \eqref{eq:Lp-u-intro} and \eqref{eq:CM} is that \eqref{eq:Lp-u-intro} permits cone singularities with the origin as the vertex of the cone for some non-symmetric $f$ (see \cite{DH22_CVPDE} for instance), whereas \eqref{eq:CM} guarantees that the origin is always located inside $\Sigma$, leading to a smooth solution \cite{HMS04_CVPDE}.

\bigskip

The fully non-linear equations in the Hessian quotient class have been investigated over a domain in $\R^n$ by Caffarelli, Nirenberg, and Spruck \cite{CNS85_Acta}, Krylov \cite{Krylov95_TAMS}, Trudinger \cite{Trudinger90_ARMA,Trudinger95_Acta}, and Bao, Chen, Guan, and Ji \cite{BCGJ03_AJM}. See also \cite{HS99_Acta} for a crucial application of the Hessian quotient in the convexity estimate of the mean convex mean curvature flow. Recently, Chen and Xu \cite{CX22_AM} addressed the quotient of mixed Hessians of the form
\begin{align*}
\fr{\si_k(\bar\mu-\mu_1,\ldots,\bar\mu-\mu_n)}{\si_{l}(\bar\mu-\mu_1,\ldots,\bar\mu-\mu_n)} = \fr{fu^{p-1}}{\lambda} \qu \text{on }\bS^n,
\end{align*}
where $\mu_1$,$\ldots$,$\mu_n$ are the eigenvalues of $\D^2u+uI$, $\bar\mu=\sum_{i=1}^n\mu_i$, 
$0\le l <k\le n$, and $p\ge k-l+1$.

Finally, we conclude the introduction by mentioning the eigenvalue problems associated with the Monge--Amp\`ere equation over a convex domain in $\R^n$, which can be found in \cite{Lions85_AMPA,Salani05_AM,Tso90_Invent}. These problems may be viewed as the Euclidean counterpart of the eigenvalue problem \eqref{eq:main-X} with $k=n$. 

\bigskip

The paper is organized as follows. In Section \ref{sec:pre} we collect relevant information, such as properties of $\si_k$, radial functions for star-shaped hypersurfaces, and support functions in convex geometry. In addition, the proof of \Cref{thm:2} is presented. In Section \ref{sec:vol-diam}, we derive estimates for volume and diameter when $p>k+1$. Section \ref{sec:curv} dedicated to establishing curvature estimates. Finally, the proof of \Cref{thm:main} is provided in Section \ref{sec:pf}.

\section{Preliminaries}\label{sec:pre}

We refer to \cite{CNS85_Acta,CNS86_incollection,CNS88_CPAM} for general background on $\si_k(\kappa)$. Recall that 
\begin{align*}
\si_k(\kappa) = \sum_{i_1<i_2<\cdots<i_k}\kappa_{i_1}\cdots\kappa_{i_k}.
\end{align*}
With an abuse of notation, we often denote by $\si_k(A)$ the $k^{\text{th}}$ elementary symmetric polynomial with respect to the eigenvalues of a symmetric two-tensor $A$. 

The elementary symmetric sum $\sigma_k$ satisfies the following property, known as the Maclaurin inequality:
\begin{align}\label{eq:Maclaurin}
c_1\si_1 \ge c_2\si_2^{1/2}\ge \cdots \ge c_n\si_n^{1/n}
\end{align}
for some dimensional constant $c_i$, $i=1,\ldots,n$.
It is also well-known that 
\begin{align}\label{eq:dsi_k}
\fr{\pa\si_k}{\pa \kappa_i}(\kappa)>0 \qu \text{for all } \kappa\in \Ga_k \text{ and }i=1,\ldots,n
\end{align}
and
\begin{align}\label{eq:si(ka)}
\si_k\left(\kappa_1, \ldots, \kappa_n\right)=\fr{\si_{n-k}(\kappa_1^{-1},\ldots,\kappa_n^{-1})}{\si_n(\kappa_1^{-1},\ldots,\kappa_n^{-1})}.
\end{align}\\

We recall some basic properties of star-shaped hypersurfaces in $\R^{n+1}$; for details, see \cite{Oliker84_CPDE}. For a star-shaped hypersurface, the position vector $X$ can be expressed as $X=\rho(x) x$ for $x\in \bS^n$ using the radial function $\rho:\bS^n\ra \R$. Then, the induced metric $g_{ij}$ is given as
\begin{align}\label{eq:g}
g_{ij}[\rho] = \rho^2 \de_{ij}+\D_i\rho\D_j\rho.
\end{align}
The unit outward normal is given by
\begin{align}\label{eq:nu}
\nu=\frac{\rho x-\nabla \rho}{\sqrt{\rho^2+|\nabla \rho|^2}},
\end{align}
and the relation between the radial function $\rho$ and the support function $u=\inn{X}{\nu}$ is
\begin{align}\label{eq:u-rho}
u=\fr{\rho^2}{\sqrt{\rho^2+|\nabla \rho|^2}}.
\end{align}
The second fundamental form can be computed as
\begin{align}\label{eq:h}
h_{i j}[\rho]=\left(\rho^2+|\nabla \rho|^2\right)^{-1 / 2}\left(\rho^2 \delta_{i j}+2 \D_i\rho\D_j\rho-\rho \D_i\D_j\rho\right).
\end{align}\\

For a moment, let us consider $\Si$ as a strictly convex, closed hypersurface in $\R^{n+1}$. Let $\nu(x)$ be the unit outer normal vector to $\Sigma$ at $x \in \Sigma$. The Gauss map $\nu$ is then a diffeomorphism from $\Sigma$ onto $\mathbb{S}^n$. The support function $u:\bS^n \ra \R$ of $\Sigma$ is defined as
\begin{align*}
u(x)=\sup\{ \inn{x}{y}: y\in \Sigma\}  = \inn{x}{ \nu^{-1}(x)}, \quad x \in \mathbb{S}^n \subset \mathbb{R}^{n+1}.
\end{align*}

One can verify that the second fundamental form $h$ is given by the expression
\begin{align*}
h_{ij}[u]= \D_i\D_j u+u\de_{ij}.
\end{align*}
Using this and the Ricci identity, we have
\begin{align*}
\D_lh_{ij}&=\D_l\D_i\D_ju+(\D_lu)\de_{ij}=\D_i\D_l\D_ju+\overline R_{lijm}\D_mu+(\D_lu)\de_{ij}
\\
&=\D_i\D_l\D_ju+(\de_{lj}\de_{im}-\de_{lm}\de_{ij})\D_mu+(\D_lu)\de_{ij}
\\
&=\D_i\D_l\D_ju+(\D_iu)\de_{lj}=\D_ih_{lj}
\end{align*}
and
\begin{align*}
\D_{k}\D_lh_{ij}&=\D_k\D_ih_{lj}=\D_i\D_kh_{lj}+\overline{R}_{kilm}h_{mj}+\overline{R}_{kijm}h_{lm}
\\
&=\D_i\D_jh_{kl}+(\de_{kl}\de_{im}-\de_{km}\de_{il})h_{mj}+(\de_{kj}\de_{im}-\de_{km}\de_{ij})h_{lm}
\\
&=\D_i\D_jh_{kl}+\de_{kl}h_{ij}-h_{kj}\de_{il}+\de_{kj}h_{il}-h_{kl}\de_{ij},
\end{align*}
where $\overline R_{ijkl}$ is the Riemannian curvature tensor on the sphere $\bS^n$.
Hence, it follows that
\begin{align}\label{eq:simons}
\De h_{ij} = \D_i\D_j \operatorname{tr}(h)+nh_{ij}-\operatorname{tr}(h)\de_{ij}.
\end{align}

The eigenvalues of $h$ correspond to the reciprocals of the principal curvatures $\kappa_i=\kappa_i[u]$ so that
\begin{align*}
\si_k(\D^2u+uI)= \si_k(\kappa_1[u]^{-1},\ldots,\kappa_n[u]^{-1}).
\end{align*}
By using \eqref{eq:si(ka)}, equation \eqref{eq:Lp} can be reformulated as 
\begin{align}\label{eq:main-u}
\fr{\det(\D^2 u+uI)}{\si_{n-k}(\D^2u+uI)} = \fr{fu^{p-1}}{\lambda} \qu \text{on }\bS^n.
\end{align}
In light of equation \eqref{eq:Lp-u}, equation \eqref{eq:main-u} migh be referred to as an $L_p$-type curvature equation.
Since $\si_k^{1/k}$ is concave operator, it is natural to define $F(A)=(\si_n(A)/\si_{n-k}(A))^{1/k}$. We then rewrite \eqref{eq:main-u} as
\begin{align*}
F(\D^2u+uI)= (f\la^{-1})^{1/k}u^{\fr{p-1}{k}} \qu \text{on }\bS^n.
\end{align*}\\

We conclude this section by establishing the existence and uniqueness results for \eqref{eq:Lp} when $k>p+1$. The following two propositions prove \Cref{thm:2}. As mentioned in the introduction, the existence result is a consequence of Theorem 1.5 in \cite{GRW15_CPAM}.

\begin{proposition}\label{prop:exist}
Let $p>k+1$, and let $f$ be a smooth positive function on $\bS^n$. Then there exists a strictly convex hypersurface $\Si$ satisfying \eqref{eq:Lp}.
\end{proposition}

\begin{proof}
It suffices to show the following (i) barrier condition and (ii) convexity property in \cite{GRW15_CPAM}: (i) there is a constant $r>1$ such that
\begin{align*}
\tilde f\left(X,\fr{X}{|X|}\right) \le \fr{\sigma_k(1,1,\ldots,1)}{r^k} \qu \text{for } |X|=r,
\end{align*}
where $\tilde f(X,\nu)= \fr{\la \psi(\nu)}{\inn{X}{\nu}^{p-1}}$; (ii) $\tilde f^{-1/k}(X,\nu)$ is locally convex in $X$ for any fixed $\nu\in\bS^n$. 

The former holds since $\tilde f(X,X/|X|)=\la \psi(\nu)r^{1-p}$ and $1-p<-k$, and the latter follows from $\tilde f^{-1/k}(X,\nu)=(\la \psi(\nu))^{-1/k}\inn{X}{\nu}^{\fr{p-1}{k}}$ and $\fr{p-1}{k}>1$.
\end{proof}

\begin{proposition}\label{prop:unique}
Suppose $p>k+1$. If $\Si_1$ and $\Si_2$ are two admissible solutions of \eqref{eq:Lp}, then $\Si_1=\Si_2$.
\end{proposition}

\begin{proof}
Let $\rho_i$, $u_i$, and $\nu_i$ ($i=1,2$) be the radial function, support function, and outward unit normal of $\Si_i$, respectively. Note that $\rho_1$ and $\rho_2$ are positive functions. Suppose that $v=\log(\rho_1/\rho_2)$ attains its maximum at $x_0\in \mathbb{S}^n$. Then, at $x_0$, we have
\begin{align*}
0=\nabla_j v&=\frac{\nabla_j \rho_1}{\rho_1}-\frac{\nabla_j \rho_2}{\rho_2},
\\
0\ge \D_i\D_j v &= \frac{\D_i\D_j \rho_1}{\rho_1}-\frac{\nabla_i\rho_1\nabla_j\rho_1}{\rho_1^2}
-\frac{\D_i\D_j \rho_2}{\rho_2}+\frac{\nabla_i\rho_2\nabla_j\rho_2}{\rho_2^2}
\\
&=\frac{\D_i\D_j \rho_1}{\rho_1}
-\frac{\D_i\D_j \rho_2}{\rho_2}.
\end{align*}
Then it follows from \eqref{eq:g} and \eqref{eq:h} that
\begin{align}\label{eq:g/rho}
\fr{g_{ij}[\rho_1]}{\rho_1^2} &= \de_{ij}+\fr{\D_i\rho_1\D_j\rho_1}{\rho_1^2}=\de_{ij}+\fr{\D_i\rho_2\D_j\rho_2}{\rho_2^2}=\fr{g_{ij}[\rho_2]}{\rho_2^2}\\
\label{eq:h/rho}
\fr{h_{ij}[\rho_1]}{\rho_1} &= \left(1+|\nabla \rho_1|^2/\rho_1^2\right)^{-1 / 2}\left(\delta_{i j}+2 \fr{\D_i\rho_1\D_j\rho_1}{\rho_1^2}-\fr{\D_i\D_j\rho_1}{\rho_1}\right)
\ge \fr{h_{ij}[\rho_2]}{\rho_2}.
\end{align}
Similarly, by utilizing homogeneity, we can obtain from \eqref{eq:nu} and \eqref{eq:u-rho} that
\begin{align}\label{eq:nu-u-rho}
\nu_1=\nu_2 \qu \text{and} \qu\fr{u_1}{\rho_1}=\fr{u_2}{\rho_2}.
\end{align}

Since the principal curvatures are the eigenvalue of $h_{ij}$ with respect to $g_{ij}$, it follows from \eqref{eq:g/rho} and \eqref{eq:h/rho} that
\begin{align*}
\rho_1\kappa_i[\rho_1]\ge \rho_2\kappa_i[\rho_2] \qu\text{for all } i=1,2,\dots,n.
\end{align*}
Therefore, we obtain
\begin{align*}
\rho_1^k\si_k(\kappa[\rho_1])\ge \rho_2^k\si_k(\kappa[\rho_2]).
\end{align*}
Using \eqref{eq:Lp} and \eqref{eq:nu-u-rho}, we get
\begin{align*}
\la u_1^{1-p}=f(\nu_1)\si_k(\kappa[\rho_1])\ge f(\nu_2)\left(\fr{\rho_2}{\rho_1}\right)^k\si_k(\kappa[\rho_2])= \la \left(\fr{u_2}{u_1}\right)^ku_2^{1-p}.
\end{align*}
This implies that $\rho_1^{k+1-p}\ge \rho_2^{k+1-p}$. Since $p>k+1$, we have $\rho_1\le \rho_2$ at $x_0$. Consequently, $\max_{\bS^n} v \le 1$. Similarly, we obtain $\min_{\bS^n} v \ge 1$, which proves that $v\equiv 1$ and, therefore, $\rho_1\equiv \rho_2$.
\end{proof}

%
%
\section{Volume and diameter estimates for $p>k+1$}\label{sec:vol-diam}
In this section, we provide estimates for volume and diameter, which will be used to extract a sequence after suitable scaling in a subsequent section. According to \Cref{thm:2}, there exists a unique, strictly convex solution $u=u_p$ to \eqref{eq:Lp} with $\lambda=1$ after appropriate dilation, provided $p>k+1$.

\begin{lemma}\label{lem:V}
Assume $p>k+1$. If $u$ is the solution to equation \eqref{eq:Lp} with $\la=1$, then for the volume $V(\Sigma)$ of the region bounded by the associated hypersurface $\Sigma$, it holds that
\begin{align}\label{eq:V-ratio}
\frac{1}{\max_{\mathbb{S}^n}f}\le \left(\fr{V(\Sigma)}{V(B_1)}\right)^{\frac{p-k-1}{n+1}} \le \frac{1}{\min_{\mathbb{S}^n}f}.
\end{align}
\end{lemma}

\begin{proof}
Let $\rho$ be the radius of the ball centered at the origin such that $V(B_\rho)=V(\Si)$. 
We first observe
\begin{align*}
\min_{\bS^n} u \le \rho \le \max_{\bS^n} u.
\end{align*}

If $u$ achieves its minimum at $x_0\in\bS^n$, then $\D^2u(x_0)\ge 0$ and consequently, $h_{ij}=\D_i\D_ju+u\de_{ij}\ge u\de_{ij}$ at $x_0$. Since the eigenvalues of $h_{ij}$ are the principal radii, we deduce $\kappa_i\le u^{-1}$ and $\si_k(\kappa)\le u^{-k}$ at the point $x_0$. It follows from \eqref{eq:Lp} that 
\begin{align*}
\fr{1}{f(x_0)}= u(x_0)^{p-1}\si_k(\kappa)(x_0)\le u(x_0)^{p-k-1} \le \rho^{p-k-1},
\end{align*}
and hence, the first inequality in \eqref{eq:V-ratio} is derived from
\begin{align*}
V(\Sigma) = V(B_\rho) =\rho^{n+1}V(B_1).
\end{align*}

The second inequality in \eqref{eq:V-ratio} can be obtained through similar arguments, by considering the maximum point of $u$ rather than the minimum point.
\end{proof}

The next lemma is valid for a more extensive range of $p$; however, the range $p>k+1$ suffices for proving the main theorem.

\begin{lemma}\label{lem:u}
Let $p>1-k/n>0$. If $u$ is the solution to \eqref{eq:Lp} with $\la=1$ and $\Si$ is the associated hypersurface, then there exists a constant $C$ that depends only on $n$, $k$, $p$, and $\min_{\bS^n}f$ such that 
\begin{align}\label{eq:max-u}
\max_{\bS^n} u \le C\left[V(\Si)\right]^{\fr{1}{p_*}},
\end{align}
where $p_*= 1+ \fr{(p-1)n}{k}>0$. Here, the constant $C$ is finite for all $p>1-k/n$ and increases in $p$
\end{lemma}

\begin{remark}
Due to the monotone increasing property of the constant in equation \eqref{eq:max-u}, it uniformly holds for $p \in (k+1, k+1+\epsilon)$, where $\epsilon > 0$ is any number.
\end{remark}

\begin{proof}
Suppose that $u$ attains its maximum at $x_0\in\bS^n$. Due to the convexity of $\Si$, the hypersurface $\Si$ encloses the line segment joining $u(x_0)x_0$ and the origin. Thus we have
\begin{align}\label{eq:u>max}
u(x)\ge \max\{\inn{u(x_0)x_0}{x},0\} =u(x_0)\max\{\inn{x_0}{x},0\}. 
\end{align}

Applying the Maclaurin inequality \eqref{eq:Maclaurin}, we deduce that $c_k\si_k^{1/k}\ge c_n\si_n^{1/n}$. This leads to
\begin{align*}
\fr{cu}{K}= \fr{cu}{\si_n}\ge \fr{u}{ \si_k^{n/k}}= u^{1+\fr{(p-1)n}{k}}f^{\fr{n}{k}}= u^{p_*}f^{\fr{n}{k}}
\end{align*}
for some dimensional constant $c>0$.
Hence, 
\begin{align}\label{eq:V>min}
c(n+1)V(\Si)= c\int_{\bS^n} \fr{u}{K} \ge \int_{\bS^n}u^{p_*}f^{\fr{n}{k}}\ge
\left(\min_{\bS^n}f\right)^{\fr{n}{k}}\int_{\bS^n}u^{p_*}.
\end{align}
Observe that $p_*=1+\fr{(p-1)n}{k}= \fr{pn-(n-k)}{k}>0$. From \eqref{eq:u>max}, we derive 
\begin{align}\label{eq:int-u^p}
\int_{\bS^n}u^{p_*}\ge u(x_0)^{p_*}\int_{\inn{x_0}{x}\ge0} \inn{x_0}{x}^{p_*},
\end{align}
where the last integral is invariant under $x_0$ due to symmetry and is a positive constant that decreases in $p_*$. The conclusion then follows from \eqref{eq:V>min} and \eqref{eq:int-u^p}.
\end{proof}

The $C^1$ estimate is derived from the convexity property.

\begin{lemma}\label{lem:C1}
Let $u$ be a smooth, strictly convex solution to \eqref{eq:main-X}. Then
\begin{align*}
|\D u| \le \max_{\bS^n} u.
\end{align*}
\end{lemma}

\begin{proof}
Since $|X|^2= u^2+|\D u|^2$, we have 
\begin{align*}
|\D u|\le \sqrt{u^2+|\D u|^2}=|X|\le \max_{\bS^n} |X|=\max_{\bS^n}u
\end{align*}
which completes the proof.
\end{proof}

%
%
\section{Uniform positive lower bound of curvature for $0<p-k-1<\e$}\label{sec:curv}

The purpose of this section is to establish that the sum of the reciprocals of the principal curvatures remains bounded by the same constant for $k+1<p<k+1+\e$, where $\e>0$ is any number.

\begin{lemma}\label{lem:1/k}
Suppose that $p> k+1$, and let $u$ be the solution of the equation \eqref{eq:Lp} with $\la=1$. Then
\begin{align*}
\sum_{i=1}^n\fr{1}{\kappa_i} \leq \fr{2n(p-1)}{k}\norm{u}_{L^\infty}^{\fr{p-1}{k}}\norm{f^\fr{1}{k}}_{C^{1,1}}.
\end{align*}
\end{lemma}

\begin{proof}
Observe that
\begin{align*}
W := \sum_{i=1}^n\frac{1}{\kappa_i} = \operatorname{tr} (h)= \Delta u + nu.
\end{align*}
We will apply the maximum principle to $W$. By recalling equation \eqref{eq:simons} and taking semi-positive definite matrix $F^{ij}=\fr{\pa F}{\pa h_{ij}}$ to both sides of the equation, we obtain
\begin{align}\label{eq:F-Deh}
F^{ij}\De h_{ij} = F^{ij}\D_i\D_j W+nF^{ij}h_{ij}-W\sum_{i=1}^nF^{ii}.
\end{align}
Here and henceforth, we employ the Einstein summation convention, in which repeated upper and lower indices are automatically summed.

Since $F$ is homogeneous of degree one, we obtain
\begin{align*}
F^{ij}h_{ij}= F= u^{\frac{p-1}{k}}f^{\fr{1}{k}}.
\end{align*}
It is known \cite{Trudinger90_ARMA} (see also \cite{GM03_Invent}) that 
\begin{align*}
\sum_{i=1}^n\fr{\pa}{\pa h_{ii}}\left(\fr{\si_n/\binom{n}{n}}{\si_{n-k}/\binom{n}{n-k}}\right)^{1/k}\ge n \qu \text{so} \qu \sum_{i=1}^nF^{ii}\ge n\left[\binom{n}{k}\right]^{-1/k}\ge 1.
\end{align*}

By taking the second derivatives to equation \eqref{eq:Lp} and utilizing the concavity of $F$, we obtain the inequality:
\begin{align}\label{eq:De-uf}
\Delta(u^{\frac{p-1}{k}}f^\fr{1}{k})=\De F= F^{ij,lm}\nabla_q h_{ij} \nabla^q h_{lm}+ F^{ij}\Delta h_{ij} \le F^{ij} \Delta h_{ij}.
\end{align}
As $p>k+1$, we have $a:= \frac{p-1}{k}>1$. Note that $\Delta u = W - nu\ge -nu$. Since
\begin{align*}
\Delta u^a = au^{a-1}\Delta u + a(a-1)u^{a-2}|\nabla u|^2\ge a u^{a-1}\Delta u=-nau^a,
\end{align*} we obtain through direct computation
\begin{align*}
\Delta (u^af^\fr{1}{k})&= f^\fr{1}{k}\Delta u^a + 2 \nabla u^a \nabla f^\fr{1}{k} + u^a \Delta f^\fr{1}{k}\\
&\ge  - 2au^{a-1}|\nabla u| |\nabla f^\fr{1}{k}|+ u^a (\Delta f^\fr{1}{k}-naf^\fr{1}{k}).
\end{align*}
Hence, combining with \eqref{eq:De-uf}, we have
\begin{align*}
F^{ij}\Delta h_{ij} \ge  - 2au^{a-1}|\nabla u| |\nabla f^\fr{1}{k}|+ u^a (\Delta f^\fr{1}{k}-naf^\fr{1}{k}).
\end{align*}
Thus, \eqref{eq:F-Deh} can be rewritten as
\begin{align*}
F^{ij}\nabla_i\nabla_j W &= F^{ij}\Delta h_{ij}-nF^{ij}h_{ij}+W\sum_{i=1}^nF^{ii}\\
&\ge  - 2au^{a-1}|\nabla u| |\nabla f^\fr{1}{k}|+ u^a (\Delta f^\fr{1}{k}-n(a+1)f^\fr{1}{k})+W.
\end{align*}
Using \Cref{lem:C1}, we deduce at the maximum point of $W$,
\begin{align*}
W \le 2au^{a-1} |\nabla u| |\nabla f^\fr{1}{k}| - u^a (\Delta f^\fr{1}{k}-n(a+1)f^\fr{1}{k}) \le 2na\lVert{u}\rVert_{L^\infty}^a \lVert{f^\fr{1}{k}}\rVert_{C^{1,1}},
\end{align*}
which completes our proof.
\end{proof}

\section{Proof of \Cref{thm:main}}\label{sec:pf}
We first prove the existence of a strictly convex solution to \eqref{eq:main-X}.
Suppose $0<p-k-1 <1$, and let $u_p$ be the unique positive solution to \eqref{eq:Lp} with $\la=1$. This is possible due to \Cref{thm:2}. Let $\Si_p$ be the associated hypersurface whose support function is $u_p$. By the uniqueness of $u_p$ and the evenness of $\psi$, we see that $\Si_p$ is origin-symmetric.

Recalling $p_*$ in \Cref{lem:u}, we see that
\begin{align*}
\fr{1}{p_*}-\fr{1}{n+1}= \fr{k}{np-n+k}-\fr{1}{n+1}=\fr{n(k+1-p)}{(np-n+k)(n+1)}.
\end{align*}
By \Cref{lem:V} and \Cref{lem:u}, we have
\begin{align}\label{eq:u_p}
u_p \le C\left[V(\Si_p)\right]^{\fr{1}{p_*}}=C\left[V(\Si_p)\right]^{\fr{1}{n+1}}\left[V(\Si_p)^{\fr{p-k-1}{n+1}}\right]^{-\fr{n}{(np-n+k)}}\le C\left[V(\Si_p)\right]^{\fr{1}{n+1}}.
\end{align}
Here we may assume the constant $C$ is uniform in $p$ since we restrict the upper bound of $p$.

Define the volume normalized solution as
\begin{align*}
\tilde u_p= \fr{u_p}{[V(\Si_p)]^\fr{1}{n+1}}
\end{align*}
and denote by $\tilde \Si_p$ the associated hypersurface of $\tilde u_p$. 
Then $\tilde u_p$ solves
\begin{align}\label{eq:main-tilde}
\fr{1}{\si_k(\kappa)} = [V(\Si_p)]^{\fr{p-k-1}{n+1}}u^{p-1}f =: \fr{u^{p-1} f }{\la_p} \qu \text{on } \bS^n.
\end{align}
Moreover, it follows from the definition of $\tilde u_p$, \eqref{eq:u_p}, and \Cref{lem:V} that
\begin{align}\label{eq:V-u-la}
V(\tilde\Si_p)=1, \qu 0<\tilde u_p \le C, \qu \text{and}\qu 1/C \le \la_p\le C.
\end{align}

We can assume that, by considering a subsequence, $\tilde u_p$ and $\la_p$ converge uniformly to a function $u_0$ and a number $\la_0$, respectively, as $p$ approaches $k+1$ from above. Note that
\begin{align*}
0\le u_0 \le C \qu \text{and} \qu 1/C\le \la_0 \le C.
\end{align*} 
Furthermore, it follows from the Blaschke selection theorem that $u_0$ is the support function of some convex hypersurface, say $\Si_0$. Since $\Si_p$ are origin-symmetric, so is $\Si_0$. Thus, if $u_0(x_0)=0$ at a point $x_0$, then we have $u_0(-x_0)=0$, which contradicts to $V( \Si_0)=1$. Therefore, $u_0>0$ on $\bS^n$.\\

Our purpose is now to show that the pair $(\Si_0,\la_0)$ is the desired solution to \eqref{eq:main-X}, which can be done by using \Cref{lem:1/k}. Indeed, let $\tilde u_{p_j}$ be a sequence such that $\tilde u_{p_j}\ra u_0$ uniformly. Since $u_0>0$, we may assume that 
\begin{align}\label{eq:u-tilde-p}
\tilde u_{p_j}\ge \fr{1}{2}\min_{\bS^n} u_0>0 \qu \text{on } \bS^n.
\end{align}

Applying \Cref{lem:1/k} to $u_{p_j}$ and using \eqref{eq:u_p}, we have
\begin{align}\label{eq:k-tilde}
\sum_{i=1}^n\fr{1}{\tilde\kappa_i}=\sum_{i=1}^n\fr{V(\Si_p)^{-\fr{1}{n+1}}}{\kappa_i} \leq CV(\Si_p)^{\fr{p-1}{k(n+1)}-\fr{1}{n+1}}=C\la_p^{-\fr{1}{k}}\le C,
\end{align}
where $\{\tilde \kappa_i\}$ are the principal curvatures of $\tilde u_{p_j}$ and the last constant $C$ does not depend on $p$. Combining \eqref{eq:main-tilde}, \eqref{eq:V-u-la}, \eqref{eq:u-tilde-p}, and \eqref{eq:k-tilde}, we obtain
\begin{align*}
\max_{i}\tilde\kappa_i \le \fr{\si_k(\tilde \kappa)}{\left(\min_i\tilde \kappa_i\right)^{k-1}}
\le 
\fr{u^{1-p}}{\la_pf\left(\min_i\tilde \kappa_i\right)^{k-1}}\le C.
\end{align*}
Therefore, $\tilde u_{p_j}$ solves the following uniformly elliptic equation
\begin{align*}
\fr{\det(\D^2 u +uI)}{\si_{n-k}(\D^2u +uI)}=\fr{ u^{p-1}f}{\la_p} \qu \text{on }\bS^n,
\end{align*}
and by the standard elliptic theory, we have $C^{3,\al}$ estimates:
\begin{align*}
\norm{\tilde u_{p_j}}_{C^{3,\al}(\bS^n)}\le C.
\end{align*}
Passing to a subsequence, we finally see that $u_0$ with $\la_0$ satisfies \eqref{eq:main-X}.

\bigskip

Lastly, we prove uniqueness. Assume that \eqref{eq:main-X} has two distinct $k$-convex solutions, namely $\left(\Si_1,\lambda_1\right)$ and $\left(\Si_2, \lambda_2\right)$. Notice that since \eqref{eq:main-X} is scale invariant, any solution multiplied by a constant continues to be a solution. By appropriately scaling $\Si_2$, we can ensure that $\Si_2$ is enclosed within $\Si_1$ and touches $\Si_1$ at a point $X_0$. We then observe that, at the point $X_0$, 
\begin{align*}
\fr{\la_1\psi(\nu_{\Si_1}(X_0))}{\inn{X_0}{\nu_{\Si_1}(X_0)}^{k}}= \si_k(\kappa[\Si_1])\le \si_k(\kappa[\Si_2])= \fr{\la_2\psi(\nu_{\Si_2}(X_0))}{\inn{X_0}{\nu_{\Si_2}(X_0)}^{k}},
\end{align*}
where we used the monotonicity \eqref{eq:dsi_k} of $\si_k$ with respect to $\kappa$. Since $\nu_{\Si_1}(X_0)=\nu_{\Si_2}(X_0)$, we conclude that $\la_1\le\la_2$. Changing the role of $\Si_1$ and $\Si_2$, we also get $\la_2\le \la_1$. Therefore, $\la_1=\la_2$.\\

Suppose we have two different solutions, $\Si_1$ and $\Si_2$, of equation \eqref{eq:main-X} with the same $\lambda$. If $\Si_2$ is not a constant multiple of $\Si_1$, we can multiply $\Si_2$ by an appropriate constant so that, after rotation if necessary, parts of $\Si_1$ and $\Si_2$ can be represented as graphs of $U_1$ and $U_2$, respectively, over the same domain $\Om$ with $U_1=U_2$ on $\pa \Om$ and $U_1>U_2$ in $\Om$. 

On the other hand, it is known that $\si_k(\kappa[\Si_i])=G(DU_i,D^2U_i)$ for some elliptic operator if $\kappa[\Si_i]\in \Ga_k$, as shown in \cite{CNS86_incollection} (see also \cite{CNS88_CPAM}). Hence, for $i=1,2$, $U_i$ satisfies 
\begin{align*}
G(DU_i,D^2U_i)=\fr{\la \psi((DU_i,-1)/\sqrt{1+|DU_i|^2})}{(x\cdot DU_i(x)-U_i)^k}
\end{align*}
since $\nu_{\Si_i}= \fr{(DU_i,-1)}{\sqrt{1+|DU_i|^2}}$ and $X(\Si_i)=(x,U_i(x))$. However, this result contradicts the maximum principle, which completes the proof.


\section*{Acknowledgement}


This work was supported by the National Research Foundation of Korea (NRF) grant funded by the Korea government (MSIT) (no. RS-2023-00211258).


\begin{thebibliography}{10}

\bibitem{BCGJ03_AJM}
{\sc Bao, J., Chen, J., Guan, B., and Ji, M.}
\newblock Liouville property and regularity of a {H}essian quotient equation.
\newblock {\em Amer. J. Math. 125}, 2 (2003), 301--316.

\bibitem{BHP18_JDG}
{\sc B\"{o}r\"{o}czky, K.~J., Henk, M., and Pollehn, H.}
\newblock Subspace concentration of dual curvature measures of symmetric convex
  bodies.
\newblock {\em J. Differential Geom. 109}, 3 (2018), 411--429.

\bibitem{BLYZ13_JAMS}
{\sc B\"{o}r\"{o}czky, K.~J., Lutwak, E., Yang, D., and Zhang, G.}
\newblock The logarithmic {M}inkowski problem.
\newblock {\em J. Amer. Math. Soc. 26}, 3 (2013), 831--852.

\bibitem{BCD17_Acta}
{\sc Brendle, S., Choi, K., and Daskalopoulos, P.}
\newblock Asymptotic behavior of flows by powers of the {G}aussian curvature.
\newblock {\em Acta Math. 219}, 1 (2017), 1--16.

\bibitem{BIS23_TAMS}
{\sc Bryan, P., Ivaki, M., and Scheuer, J.}
\newblock Christoffel-{M}inkowski flows.
\newblock {\em Trans. Amer. Math. Soc. 376}, 4 (2023), 2373----2393.

\bibitem{CNS85_Acta}
{\sc Caffarelli, L., Nirenberg, L., and Spruck, J.}
\newblock The {D}irichlet problem for nonlinear second-order elliptic
  equations. {III}. {F}unctions of the eigenvalues of the {H}essian.
\newblock {\em Acta Math. 155}, 3-4 (1985), 261--301.

\bibitem{CNS86_incollection}
{\sc Caffarelli, L., Nirenberg, L., and Spruck, J.}
\newblock Nonlinear second order elliptic equations. {IV}. {S}tarshaped compact
  {W}eingarten hypersurfaces.
\newblock In {\em Current topics in partial differential equations}.
  Kinokuniya, Tokyo, 1986, pp.~1--26.

\bibitem{CNS88_CPAM}
{\sc Caffarelli, L., Nirenberg, L., and Spruck, J.}
\newblock Nonlinear second-order elliptic equations. {V}. {T}he {D}irichlet
  problem for {W}eingarten hypersurfaces.
\newblock {\em Comm. Pure Appl. Math. 41}, 1 (1988), 47--70.

\bibitem{CX22_AM}
{\sc Chen, C., and Xu, L.}
\newblock The {$L_p$} {M}inkowski type problem for a class of mixed {H}essian
  quotient equations.
\newblock {\em Adv. Math. 411}, part A (2022), Paper No. 108794, 27.

\bibitem{CD21_Duke}
{\sc Choi, B., and Daskalopoulos, P.}
\newblock Evolution of noncompact hypersurfaces by inverse mean curvature.
\newblock {\em Duke Math. J. 170}, 12 (2021), 2755--2803.

\bibitem{CW06_AM}
{\sc Chou, K.-S., and Wang, X.-J.}
\newblock The {$L_p$}-{M}inkowski problem and the {M}inkowski problem in
  centroaffine geometry.
\newblock {\em Adv. Math. 205}, 1 (2006), 33--83.

\bibitem{DH22_CVPDE}
{\sc Daskalopoulos, P., and Huisken, G.}
\newblock Inverse mean curvature evolution of entire graphs.
\newblock {\em Calc. Var. Partial Differential Equations 61}, 2 (2022), Paper
  No. 53.

\bibitem{GLM18_NoDEA}
{\sc Gao, S., Li, H., and Ma, H.}
\newblock Uniqueness of closed self-similar solutions to
  {$\sigma_k^{\alpha}$}-curvature flow.
\newblock {\em NoDEA Nonlinear Differential Equations Appl. 25}, 5 (2018),
  Paper No. 45, 26.

\bibitem{GG02_Annals}
{\sc Guan, B., and Guan, P.}
\newblock Convex hypersurfaces of prescribed curvatures.
\newblock {\em Ann. of Math. (2) 156}, 2 (2002), 655--673.

\bibitem{GLL12_Duke}
{\sc Guan, P., Li, J., and Li, Y.}
\newblock Hypersurfaces of prescribed curvature measure.
\newblock {\em Duke Math. J. 161}, 10 (2012), 1927--1942.

\bibitem{GLM06_CAMSB}
{\sc Guan, P., Lin, C., and Ma, X.}
\newblock The {C}hristoffel-{M}inkowski problem. {II}. {W}eingarten curvature
  equations.
\newblock {\em Chinese Ann. Math. Ser. B 27}, 6 (2006), 595--614.

\bibitem{GLM09_IMRN}
{\sc Guan, P., Lin, C., and Ma, X.-N.}
\newblock The existence of convex body with prescribed curvature measures.
\newblock {\em Int. Math. Res. Not. IMRN}, 11 (2009), 1947--1975.

\bibitem{GM03_Invent}
{\sc Guan, P., and Ma, X.-N.}
\newblock The {C}hristoffel-{M}inkowski problem. {I}. {C}onvexity of solutions
  of a {H}essian equation.
\newblock {\em Invent. Math. 151}, 3 (2003), 553--577.

\bibitem{GMZ06_CPAM}
{\sc Guan, P., Ma, X.-N., and Zhou, F.}
\newblock The {C}hristofel-{M}inkowski problem. {III}. {E}xistence and
  convexity of admissible solutions.
\newblock {\em Comm. Pure Appl. Math. 59}, 9 (2006), 1352--1376.

\bibitem{GRW15_CPAM}
{\sc Guan, P., Ren, C., and Wang, Z.}
\newblock Global {$C^2$}-estimates for convex solutions of curvature equations.
\newblock {\em Comm. Pure Appl. Math. 68}, 8 (2015), 1287--1325.

\bibitem{GX18_CVPDE}
{\sc Guan, P., and Xia, C.}
\newblock {$L^p$} {C}hristoffel-{M}inkowski problem: the case {$1<p<k+1$}.
\newblock {\em Calc. Var. Partial Differential Equations 57}, 2 (2018), Paper
  No. 69, 23.

\bibitem{GLW22_arxiv}
{\sc Guang, Q., Li, Q.-R., and Wang, X.-J.}
\newblock The $L_p$-{M}inkowski problem with super-critical exponents.
\newblock {Preprint \href{https://arxiv.org/abs/2203.05099}{arXiv:2203.05099}} (2022).

\bibitem{HMS04_CVPDE}
{\sc Hu, C., Ma, X.-N., and Shen, C.}
\newblock On the {C}hristoffel-{M}inkowski problem of {F}irey's {$p$}-sum.
\newblock {\em Calc. Var. Partial Differential Equations 21}, 2 (2004),
  137--155.

\bibitem{Huisken90_JDG}
{\sc Huisken, G.}
\newblock Asymptotic behavior for singularities of the mean curvature flow.
\newblock {\em J. Differential Geom. 31}, 1 (1990), 285--299.

\bibitem{HS99_Acta}
{\sc Huisken, G., and Sinestrari, C.}
\newblock Convexity estimates for mean curvature flow and singularities of mean
  convex surfaces.
\newblock {\em Acta Math. 183}, 1 (1999), 45--70.

\bibitem{Krylov95_TAMS}
{\sc Krylov, N.~V.}
\newblock On the general notion of fully nonlinear second-order elliptic
  equations.
\newblock {\em Trans. Amer. Math. Soc. 347}, 3 (1995), 857--895.

\bibitem{Lions85_AMPA}
{\sc Lions, P.-L.}
\newblock Two remarks on {M}onge-{A}mp\`ere equations.
\newblock {\em Ann. Mat. Pura Appl. (4) 142\/} (1985), 263--275 (1986).

\bibitem{Lutwak93_JDG}
{\sc Lutwak, E.}
\newblock The {B}runn-{M}inkowski-{F}irey theory. {I}. {M}ixed volumes and the
  {M}inkowski problem.
\newblock {\em J. Differential Geom. 38}, 1 (1993), 131--150.

\bibitem{LO95_JDG}
{\sc Lutwak, E., and Oliker, V.}
\newblock On the regularity of solutions to a generalization of the {M}inkowski
  problem.
\newblock {\em J. Differential Geom. 41}, 1 (1995), 227--246.

\bibitem{LYZ18_AM}
{\sc Lutwak, E., Yang, D., and Zhang, G.}
\newblock {$L_p$} dual curvature measures.
\newblock {\em Adv. Math. 329\/} (2018), 85--132.

\bibitem{McCoy11_ASNSPCS}
{\sc McCoy, J.~A.}
\newblock Self-similar solutions of fully nonlinear curvature flows.
\newblock {\em Ann. Sc. Norm. Super. Pisa Cl. Sci. (5) 10}, 2 (2011), 317--333.

\bibitem{Oliker84_CPDE}
{\sc Oliker, V.~I.}
\newblock Hypersurfaces in {${\bf R}\sp{n+1}$} with prescribed {G}aussian
  curvature and related equations of {M}onge-{A}mp\`ere type.
\newblock {\em Comm. Partial Differential Equations 9}, 8 (1984), 807--838.

\bibitem{Salani05_AM}
{\sc Salani, P.}
\newblock A {B}runn-{M}inkowski inequality for the {M}onge-{A}mp\`ere
  eigenvalue.
\newblock {\em Adv. Math. 194}, 1 (2005), 67--86.

\bibitem{Schneider14_book}
{\sc Schneider, R.}
\newblock {\em Convex bodies: the {B}runn-{M}inkowski theory}, expanded~ed.,
  vol.~151 of {\em Encyclopedia of Mathematics and its Applications}.
\newblock Cambridge University Press, Cambridge, 2014.

\bibitem{Trudinger90_ARMA}
{\sc Trudinger, N.~S.}
\newblock The {D}irichlet problem for the prescribed curvature equations.
\newblock {\em Arch. Rational Mech. Anal. 111}, 2 (1990), 153--179.

\bibitem{Trudinger95_Acta}
{\sc Trudinger, N.~S.}
\newblock On the {D}irichlet problem for {H}essian equations.
\newblock {\em Acta Math. 175}, 2 (1995), 151--164.

\bibitem{Tso90_Invent}
{\sc Tso, K.}
\newblock On a real {M}onge-{A}mp\`ere functional.
\newblock {\em Invent. Math. 101}, 2 (1990), 425--448.

\end{thebibliography}

\end{document}